\documentclass[11pt,reqno]{amsart}
\usepackage[utf8]{inputenc}

\usepackage{amsfonts, amsthm, amsmath, amssymb,bm}
\usepackage{amssymb,amscd}
\usepackage{mathtools}

\usepackage{enumerate}
\usepackage{cases}

\newcommand{\rmod}{\!\!\!\!\pmod}

\newcommand{\Aa}{\mathbb{A}}
\newcommand{\Ff}{\mathbb{F}}
\newcommand{\Gg}{\mathbb{G}}
\newcommand{\Zz}{\mathbb{Z}}

\newtheorem{thm}{Theorem}[section]
\newtheorem{lem}[thm]{Lemma}

\newtheorem{prop}[thm]{Proposition}
\newtheorem{cor}[thm]{Corollary}

\title[The twelfth moment of Dirichlet $L$-functions]{The twelfth moment of Dirichlet $L$-functions with smooth moduli}
\author{Ramon M. Nunes} \thanks{This work is supported by the DFG-SNF lead agency program grant 200021L-153647.}
\address{
   R. M. Nunes // EPFL SB MATH TAN // Station 8// CH-1015 Lausanne, Switzerland}
   \email{ramon.moreiranunes@epfl.ch}

\date{\today}

\usepackage[square,sort,comma,numbers]{natbib}

\begin{document}

\begin{abstract}
We prove an analogue of Heath-Brown's bound on the twelfth moment of the Riemann zeta function for Dirichlet $L$-functions with smooth moduli.
\end{abstract}

\maketitle

\section{Introduction}

The so-called Weyl bound for the Riemann zeta function states that
\begin{equation}\label{weyl}
\zeta(1/2+it)\ll_{\epsilon}t^{1/6+\epsilon},
\end{equation}
where here and throughout the paper, $\epsilon$ denotes an arbitrarily small constant that may vary at each ocurrence and the implied constants depend at most on the variables in the subscript.

This bound was improved over the years but not significantly so. Indeed the best result to date is the very recent result of Bourgain \citep{Bourgain2017decoupling}:
$$
\zeta(1/2+it)\ll_{\epsilon}t^{13/84+\epsilon}.
$$
Notice that $\frac{13}{84}=\frac{1}{6}-\frac{1}{84}$. Through this perspective, Heath-Brown \citep{Heath1978twelfth} proved a very interesting estimate:
$$
\int_T^{2T}|\zeta(1/2+it)|^{12}dt\ll_{\epsilon} T^{2+\epsilon}.
$$
This bounds, combined with some regularity properties of $\zeta$ not only recover \eqref{weyl} but it also proves that $|\zeta(1/2+it)|$ cannot be 'large' very often. In other words, one has
\begin{equation}\label{ineq-for-measure}
\mu(\{t\in[T,2T];\;|\zeta(1/2+it)|>V\})\ll_{\epsilon} T^{2+\epsilon}V^{-12},
\end{equation}
where here $\mu$ denotes the Lebesgue measure. Notice that the classical bound for the fourth moment of $\zeta$ gives \eqref{ineq-for-measure} with $T^{1+\epsilon}V^{-4}$ instead of $T^{2+\epsilon}V^{-12}$. In view of that and Bourgain's bound, the interest of \eqref{ineq-for-measure} lies in the range $T^{1/8}\leq V\leq T^{13/84}$.

Whenever we have an asymptotic result about $\zeta(1/2+it)$ as $t\rightarrow +\infty$, it is natural to ask the analogous question for the Dirichlet $L-$functions $L(1/2,\chi)$ as $q\rightarrow+\infty$, where $\chi$ is a primitive character modulo $q$. Unfortunately, in this setting, we do not even know the analogue of \eqref{weyl}. In full generality, the best result is due to Burgess and gives
\begin{equation}\label{burgess}
L(1/2,\chi)\ll_{\epsilon}q^{3/16+\epsilon}.
\end{equation}
The situation becomes better if one restricts their attention to moduli $q$ having a certain type of factorization. Indeed, Heath-Brown \citep{Heath1978hybrid} proved that if $q$ has a divisor $q_0$ such that $q^{1/3}\leq q_0<q^{1/3+\epsilon}$, then one has the inequality
$$
L(1/2,\chi)\ll_{\epsilon}q^{1/6+\epsilon}.
$$
In particular, this includes the case where $q=p^n$ for a small prime $p$ and large $n$ and that of $q^{\epsilon}$-smooth numbers. We recall that a number $q$ is said to be $y-$smooth if every divisor of $q$ is smaller than $y$. As for $\zeta$, one can go below $1/6$ but not by a significant amount. Indeed, let $q$ be either $q^{\delta}$-smooth or $q=p^n$, where both $\delta^{-1}$ and $n$ are sufficientely large in terms of $\epsilon$. Then one has the inequality
$$
L(1/2,\chi)\ll_{\epsilon} q^{\theta_0+\epsilon},\:\:\:\:\:\:\theta_0=0.1645\ldots
$$
For proofs, we refer the intererested reader to \citep{Irving2016estimates} and \citep{WUXI} for the case of smooth numbers and to \citep{Milicevic2016subweyl} and \citep{MunshiSingh} for the case of prime powers.

Finally, we would like to mention another important case where one is able to improve upon Burgess' bound, the case of real characters. Indeed, let $q$ be an odd squarefree number and let $\chi_q$ be the unique non-trivial primitive real character modulo $q$. then Conrey and Iwaniec \citep{CI2000cubic} proved the bound
$$
L(1/2,\chi_q)\ll_{\epsilon}q^{1/6+\epsilon}.
$$
The proof of this result is more involved than the previously mentioned ones as it employs the spectral theory of automorphic forms on $\mathrm{GL}_2$ and ultimately relies in a deep result of Waldspurger on the positivity of some $\mathrm{GL}_2$ $L$-functions. Unlike the cases discussed before, improving the exponent $1/6$ in this result seems out of reach of currently known methods.

We shall now turn to our results. Our main theorem is an analogue of Heath-Brown's bound on the twelfth moment for $L$-functions with moduli $q$ that are $q^{\delta}$-smooth for small $\delta>0$.

\begin{thm}\label{12th}
Let $\delta>0$ and suppose $q$ is a $q^{\delta}$-smooth squarefree number, then one has the inequality
$$
\sum_{\chi\rmod q}|L(1/2,\chi)|^{12}\ll_{\delta} q^{2+O(\delta)}.
$$
\end{thm}
Integration by parts shows that this follows from a bound for the number of $\chi$ for which $L(1/2,\chi)$ is large. This is contained in the following:

\begin{thm}\label{large}
Let $V,\delta>0$ and suppose $q$ is a $q^{\delta}$-smooth squarefree number. Then one has the inequalities

\begin{subnumcases}{\#\mathcal{R}(V;q)\ll_{\delta}q^{O(\delta)}}
qV^{-4},\; (V>0),\label{A1}\\
q^{2}V^{-12},\; (V>0),\label{A2}\\
qV^{-6},\; (V>q^{2/13}),\label{A3}\\
q^{5}V^{-32},\; (q^{2/13}\geq V> q^{3/20})\label{A4},
\end{subnumcases}
where $\mathcal{R}(V;q):=\{\chi\text{ primitive of modulus }q;\;|L(1/2,\chi)|>V\}$.
\end{thm}

\subsection{Remarks}

\begin{enumerate}[(i)]

\item  Theorem \ref{12th} follows directly from \eqref{A2}. The bounds \eqref{A3} and \eqref{A4} give improvements of this result for $V>q^{3/20}$. If one could prove \eqref{A3} for every $V>0$ this would give an almost sharp bound for the sixth moment of Dirichlet $L$ functions. Unfortunately this result is a challenging open problem even in the more classical setting of the $\zeta$-function.\\

\item Perhaps the most straightforward generalization of Heath-Brown's result would be to bound the twelfth moment of Dirichlet $L$-functions in the $t$ aspect. This was considered by Meurman \citep{Meurman1984mean} and then by Jutila and Motohashi \citep{JM1995mean}. The main result of the latter gives the bound

$$
\sum_{q\leq Q}\;\sideset{}{{}^{\ast}}\sum_{\chi\rmod{q}}\int_{T}^{2T}\left|L((1/2+it,\chi)\right|^{12}dt\ll_{\epsilon} Q^3T^2(QT)^{\epsilon}.
$$
Note that from this bound, we can deduce the estimate
$$
L(1/2+it,\chi)\ll_{\epsilon}q^{1/4+\epsilon}t^{1/6+\epsilon},
$$
which is as good as \eqref{weyl} with respect to $t$, but worse than \eqref{burgess} with respect to $q$. It is important to recall that this result works for general $q$ so that so that it would be too much to expect anything smaller than $3/16+\epsilon$ as the exponent of $q$ in the above inequality.\\

\item The idea of taking advantage of the factorization of $q$ to obtain stronger estimates for exponential sums is know to specialists as the \textit{$q$-van der Corput mehtod} in analogy to the classical van der Corput method for analytic exponential sums. It has been around at least since \citep{Heath1978hybrid} and had many applications over the years. A very nice general method can be found in \citep{Polymath} and \citep{WUXI}.\\

\item It is conceivable that one can prove Theorems \ref{12th} and \ref{large} when $q=p^n$, where $p$ is a fixed prime and $n\gg \delta^{-1}$. The general set-up would be very similar but the techniques to manipulate the exponential sums would be very different. In that case, there is a classical method for explicitely evaluating exponential sums but the caclulations can rapidly get messy.

\end{enumerate}

\subsection{Overview of the proof and analogy with $\zeta$.}
In the following we are a little imprecise, focusing only on the ideas. For example we will completely ignore the ubiquitous $q^{\epsilon}$-factors.

We start by supposing we can write $q=q_1Q_1$, we fix some character $\psi_1$ modulo $Q_1$ and consider the short moment

$$
S_2(\psi_1):=\sideset{}{{}^{\ast}}\sum_{\chi_1\rmod{q_1}}|L(1/2,\chi_1\psi_1)|^2.
$$
In Proposition \ref{openup+poisson} we use the approximate functional equation and Poisson summation in order to prove an upper bound for $S_2(\psi_1)$ of the shape
$$
S_2(\psi_1)\ll q_1\left\{1+\frac{1}{Q_1^{1/2}}\left|\sum_{m\leq q/q_1^2}\alpha_mK(\psi_1,m)\right|\right\},
$$
where $\alpha_m$ is essentially bounded and $K(\psi_1,.)$ is and algebraic oscillating function defined modulo $Q_1$. This is analogous to \citep[Lemma 1]{Heath1978twelfth} where an estimate is given for
$$
S_2(T_0,G):=\int_{T_0-G}^{T_0+G}|\zeta(1/2+it)^2|dt,
$$
for some $G\leq T$, $T_0\sim T$, that roughly looks like
$$
G\left\{1+\left(\frac{G}{T}\right)^{1/2}\left|\sum_{m\leq T/G^2}\alpha_me(f(T_0,m))\right|\right\},
$$ 
where $\alpha_m$ is essentially bounded and $f$ is a smooth function such that
$$
\frac{d^j}{dx^j}f(T_0,x)\asymp T_0G^{-1}\left(\frac T{G^2}\right)^{-j}.
$$
One should remark the pleasing analogy between the tuples $(G,T_0,e(f(T_0,.)))$ and $(q_1,\psi_1,K(\psi_1,.))$.

In the following we suppose that we can further factor $Q_1=q_2q_3$, and we fix a primitive character $\chi_3$ of modulus $q_3$ and consider
$$
\sum_{\chi_2\in\mathfrak{X}}S_2(\chi_2\chi_3),
$$
where $\mathfrak{X}\in\widehat{\left(\Zz/q\Zz\right)^*}$ is any set such that $\#\mathfrak{X}=X$ . By means of Cauchy-Schwarz and estimates for incomplete exponential sums (Lemma \ref{PVKK}) we arrive at a bound that looks like
$$
\sum_{\chi_2\in\mathfrak{X}}S_2(\chi_2\chi_3)\ll (q_1+\Xi(q,q_1,q_2))X+(q/q_1)^{1/2}X^{1/2},
$$
where $\Xi(q,q_1,q_2)=q_1^{1/2}q_2^{1/4}$ or $\Xi(q,q_1,q_2)=q^{1/4}q_2^{1/12}$, if $q_2>q^{3/2}q_1^{-3}$. This has a parallel with \citep[Lemma 2]{Heath1978twelfth}, which basically implies that
$$
\sum_{i=1}^{X}S_2(T_i,G)\ll (G+\Xi(G,J,T))X+(T/G)^{1/2}X^{1/2},
$$
with $\Xi(G,J,T)=G^{1/4}J^{1/4}$, or $\Xi(G,J,T)=G^{-1/12}J^{1/12}T^{1/4}$ if $J>T^{3/2}G^{-2}$, where the $T_i's$ satisfy
$$
G\ll |T_i-T_j|\ll J,\,\,\,\text{ for }i\neq j.
$$
Again, it is instructive to consider the dictionary between the tuples
$$
(G,J/G,\{T_i\}_{i=1}^{X})\;\;\;\text{    and    }\;\;\;(q_1,q_2,\mathfrak{X}).
$$
It is now a question of choosing parameters to deduce Theorem \ref{large}. At this point we make crucial use of the smoothness of $q$.
 
\section*{Acknowledgements}

I am very thankful to Philippe Michel for many fruitful discussions on the subject of this paper, and especially for his kind explanations on the formalism of trace functions.

\section{Preliminary results}

\subsection{On certain exponential sums}

The proof of Theorem \ref{large} relies on bounds for exponential sums coming from the works of Weil and Deligne. In fact all the exponential sums we will encounter will be constructed from the following one: For an integer $q$, a character $\chi$ of modulus $q$ and $k,\ell,m \in \Zz/q\Zz$, we let
\begin{equation}\label{Kchi}
K_{\chi}(k,\ell):=q^{-1/2}\sideset{}{{}^{\ast}}\sum_{u\rmod {q}}\chi(\ell+u)\overline{\chi}(u)e\left(\frac{ku}{q}\right),\;\;\;K_{\chi}(m):=K_{\chi}(m,1).
\end{equation}
Our first lemma says that $K_{\chi}(k,\ell)$ essentially depends on the product $k\ell$. The proof is a simple consequence of the Chinese remainder theorem and a trivial explicit evaluation of $K_{\chi}(k,\ell)$ when $k$ or $\ell$ equal $0$.

\begin{lem}\label{k*l}
Let $q$ be squarefree and suppose $\chi$ is a primitive character modulo $q$. And let $d=(\ell,q)$. Then we have
$$
K_{\chi}(k,\ell)=\mu(d)R_d(k)K_{\chi}(k\ell),
$$
where $R_d$ is the Ramanujan sum given by
$$
R_d(k):=\sideset{}{{}^{\ast}}\sum_{x\rmod{d}}e\left(\frac{kx}{d}\right).
$$
In particular, $|K_{\chi}(0,\ell)|\leq (\ell,q)q^{-1/2}$.
\end{lem}

The next lemma contains some basic facts about $K_{\chi}$. the first one is twisted multiplicativity and is a consequence of the Chinese remainder theorem and the second one is a uniform upper bound furnished by the work of Weil.

\begin{lem}\label{lemKchi}
Suppose $q=q_1q_2$ with $(q_1,q_2)=1$ and, let $m\in \Zz/q\Zz$  let $K_{\chi}(m)$ be as defined in \eqref{Kchi}. Then, if $\chi=\chi_1\chi_2$ with $\chi_i$ a primitive character modulo $q_i$, we have
$$
K_{\chi}(m)=K_{\chi_1}(\overline{q_2}m)K_{\chi_2}(\overline{q_1}m).
$$
Moreover, for every squarefree number $q$, every primitive character $\chi\pmod q$ and every $m$ modulo $q$,
$$
K_{\chi}(m)\ll q^{\epsilon}.
$$
\end{lem}

As we mentioned above, we will construct other exponential sums from $K_{\chi}$. We will be particularly interested in the function
$$
m\mapsto K_{\chi}(m)\overline{K_{\chi'}(m)},
$$
where $\chi$ and $\chi'$ are two primitive Dirichlet characters. Suppose first that $\chi$ and $\chi'$ are characters of prime modulus $p$. There are two possibilities. Either $\chi\neq \chi'$ and we have a nice oscillating function or $\chi=\chi'$ is which case we have a function that is always non-negative. Fortunately, substracting the constant function $1$ provides us with a function suiting our needs. We let

\begin{equation}\label{Kcirc}
K^{\circ}_{\chi,\chi'}(m):=K_{\chi}(m)\overline{K_{\chi'}(m)}-\delta_{\chi,\chi'},
\end{equation}
where $\delta_{\chi.\chi'}$ equals $1$ if $\chi=\chi'$ and $0$ otherwise.

Our next Proposition summarizes what we mean by $K^{\circ}_{\chi,\chi'}$ being a nice oscillating function. Since its proof uses such different techniques from those in the rest of the paper, we postpone its proof until section \ref{tracefunctions}.

\begin{prop}\label{completesums}
Let $p$ be a prime number, let $\chi,\chi'$ be primitive Dirichlet characters modulo $p$ and let $K^{\circ}_{\chi,\chi'}$ be as in \eqref{Kchi}. Let $t\in\Zz/p\Zz$. Then we have

\begin{equation}\label{PV}
\sum_{u\rmod p}K^{\circ}_{\chi,\chi'}(u)e\left(-\frac{tu}{p}\right)=
\begin{cases}
-1,\text{ if }t=0,\\
O(p^{1/2}),\text{ if }t\neq0.
\end{cases}
\end{equation}
Moreover, let $s\in\Zz/p\Zz$, then we have
\begin{align}\label{vdC}
\sum_{u\rmod p}K^{\circ}_{\chi,\chi'}(s+u)\,\overline{K^{\circ}_{\chi,\chi'}(u)}e\left(-\frac{tu}{p}\right)\notag\\
\ll
\begin{cases}
p,\text{ if }t=s=0,\\
p^{1/2},\text{ otherwise}.
\end{cases}
\end{align}

\end{prop}
In the following we extend the definition of $K^{\circ}_{\chi,\chi'}$ to characters with squarefree moduli, and we do so in order to keep the twisted multiplicativity of Lemma \ref{lemKchi}. Suppose $\chi=\prod_{p\mid q}\chi_p$ and $\chi'=\prod_{p\mid q}\chi'_p$ then we let
\begin{equation}\label{twistedKcirc}
K^{\circ}_{\chi,\chi'}(m):=\prod_{p\mid q}K^{\circ}_{\chi_p,\chi'_p}\left(\overline{Q_p}m\right),\;\;\;Q_p:=q/p.
\end{equation}
The main reason for this definition is because it recovers $K_{\chi}(m)\overline{K_{\chi'}(m)}$, when $\chi_p\neq \chi'_p$ for all $p\mid q$ and more generally, we have
\begin{equation}\label{Lchi}	
K_{\chi}(m)\overline{K_{\chi'}(m)}:=\displaystyle\sum_{\substack{q=rs\\\Delta\mid r}} K^{\circ}_{\chi_{[r]},\chi'_{[r]}}(\overline{s}m),
\end{equation}
where $\chi_{[r]}=\prod_{p\mid r}\chi_p$, $\chi'_{[r]}$ is defined analogously, and, finally, $\Delta=\Delta(\chi,\chi')$ is the \textit{distance between the two characters} given by
$$
\Delta(\chi,\chi'):=\prod\left\{p;p\mid q,\,\chi_p\neq \chi'_p\right\}.
$$
The bounds for complete sums from Proposition \ref{completesums} first come into play by means of the Polya-Vinogradov completion method, which basically means detecting a congruence by additive characters and using classical bounds for a sum of a geometric sequence. We summarize this in the next lemma.
\begin{lem}\label{PVKK}
Let $q$ be squarefree and let $\chi$ and $\chi'$ denote two characters modulo $q$. Let $r$ be coprime to $q$. Then one has the inequality
$$
\sum_{m\leq M}K^{\circ}_{\chi,\chi'}(mr)\ll q^{\epsilon}(q^{-1}M+q^{1/2}).
$$
\end{lem}

\begin{proof}

We first separate the sum in different congruence classes and then use orthogonality of characters to detect the congruence condition. Since $r$ is coprime to $q$,

\begin{multline*}
\sum_{m\leq M}K^{\circ}_{\chi,\chi'}(mr)=\sum_{u\rmod q}K^{\circ}_{\chi,\chi'}(u)\sum_{\substack{m\leq M\\ m\equiv \overline{r}u\rmod q}}1\\
=\frac1q \sum_{t\rmod q}\left(\sum_{u\rmod q}K^{\circ}_{\chi,\chi'}(u)e\left(-\frac{\overline{r}tu}{q}\right)\right)\sum_{\substack{m\leq M\\ m\equiv u\rmod q}}e\left(\frac{mt}{q}\right).
\end{multline*}

Suppose $\chi=\prod_{p}\chi_p$ and $\chi'=\prod_{p}\chi_p'$. Now by twisted multiplicativity (Lemma \ref{lemKchi}) and the Chinese remainder theorem,

\begin{multline*}
\sum_{u\rmod q}K^{\circ}_{\chi,\chi'}(u)e\left(-\frac{\overline{r}tu}{q}\right)=\prod_{p\mid q}\left(\sum_{u\rmod p}K^{\circ}_{\chi_p,\chi'_p}(u)e\left(-\frac{\overline{r}tu}{p}\right)\right).
\end{multline*}

In particular, Lemma \ref{PVKK} implies that

$$
\left(\sum_{u\rmod q}K^{\circ}_{\chi,\chi'}(u)e\left(-\frac{\overline{r}tu}{q}\right)\right)\ll
\begin{cases}
1,\text{ if }t=0,\\
q^{1/2}(t,q)^{1/2},\text{ if }t\neq 0.
\end{cases}
$$

Finally we have the following bound for the sum of a geometric sequence:

$$
\sum_{m\leq M}e\left(\frac{mt}{q}\right)\ll \min\left(M,\left\|\frac tq\right\|^{-1}\right),
$$
where $\|.\|$ denotes distance to the closest integer. Putting everything together we obtain that

$$
\sum_{m\leq M}K^{\circ}_{\chi,\chi'}(mr)\ll q^{-1}M+\frac{1}{q^{1/2}}\sum_{t=1}^{q-1}(t,q)^{1/2}\left\|\frac tq\right\|^{-1}.
$$
The lemma now follows.
\end{proof}

It is a well-know feature of the Polya-Vinogradov method that it can only give non-trivial results for sums that are longer than the square-root of the conductor. We can however get by if we assume that we can factor $q$ in a suitable way. This is known to experts as the $q$-van der Corput method.

\begin{lem}\label{vdCKK}
Let $q=q_1q_2$ be squarefree and let $\chi$ and $\chi'$ denote two character modulo $q$. Let $r$ be comprime to $q$. Then one has the inequality

$$
\sum_{m\leq M}K^{\circ}_{\chi,\chi'}(mr)\ll q^{\epsilon}(Mq_1^{-1/4}+M^{1/2}q_2^{1/2}+M^{1/2}q_1^{1/4}).
$$

\end{lem}

\begin{proof}
Let $F(m):=K^{\circ}_{\chi,\chi'}(mr)$ and $S(M):=\sum_{m\leq M}F(m)$. From \eqref{twistedKcirc}, one sees that $F(m)=F_1(\overline{q_2}m)F_2(\overline{q_1}m)$, where $\chi=\chi_1\chi_2$ $\chi'=\chi'_1\chi'_2$ and $F_i(m):=K^{\circ}_{\chi_i,\chi'_i}(mr)$.
We start by splitting our sum acording to the congruence class of $m$ modulo $q_2$.
$$
S(M)=\sum_{u\rmod{q_2}}F_2(\overline{q_1}u)\sum_{\substack{m\leq M\\m\equiv u\rmod {q_2}}}F_1(\overline{q_2}m).
$$
Now by Cauchy-Schwarz and Lemma \ref{lemKchi}, we obtain
$$
|S(M)|^2\ll q_2^{1+\epsilon}\sum_{\substack{m,m'\leq M\\m\equiv m'\rmod{q_2}}}F_1(\overline{q_2}m)\overline{F_1(\overline{q_2}m)}.
$$
Expanding the inner sum and estimating trivially the contribution from the diagonal, one gets
\begin{equation}\label{eqVDC}
|S(M)|^2\ll q_2^{1+\epsilon}\left(M+\sum_{0<|\ell|\leq M/q_2}\sum_{m\in I_{M,\ell q}}F_1(\overline{q_2}m)F_1(\overline{q_2}m+\ell)\right),
\end{equation}
where $I_{M,\ell q_2}\subset (0,M]$ is an interval. Now the method of the previous lemma together with \eqref{vdC} gives

$$
\sum_{m\in I_{M,\ell q}}F_1(\overline{q_2}m)F_1(\overline{q_2}m+\ell)\ll Mq_1^{-1/2}(q_1,\ell)^{1/2}+q_1^{1/2+\epsilon}.
$$

Now summing over $\ell$ and taking the square root on both sides of \eqref{eqVDC} gives the result.
\end{proof}

\begin{cor}
With the assumptions as in Lemma \ref{vdCKK}, suppose further that $q=q_1q_2$ and that $q^{\frac 23}y^{-\frac 13}<q_1\leq q^{\frac 23}y^{\frac 13}$. Then we have
$$
\sum_{m\leq M}K^{\circ}_{\chi,\chi'}(mr)\ll_{\epsilon} q^{\epsilon}\left(Mq^{-1/6}y^{1/12} + M^{1/2}q^{1/6}y^{1/6}\right).
$$
\end{cor}

\subsection{Approximate functional equation}\label{AFE}

As is usual when dealing with $L$-functions, we shall use a version of the approximate functional equation. Since in our work we are only concerned with upper bounds, we will avoid the often cumbersome problems arising from the oscilations of the $\epsilon$-factor (in the case of Dirichlet $L$-functions, these are normalized Gauss sums).

The following lemma follows directly from applying the approximate functional eqaution from \citep[Theorem 5.3]{IwKo2004analytic} followed by a classical dyadic decomposition. The claimed  properties of the functions $V^{\pm}_N$ follow from \citep[Theorem 5.4]{IwKo2004analytic}.

Finally, we use the notation $\displaystyle\sum_{\substack{N\leq X\\ N\text{dyadic}}}$ to indicate a sum over positive powers of two that are $\leq X$.

\begin{lem}\label{AFE-lem} Let $q$ be a positive integer and let $\chi$ be a primitive character modulo $q$. Then we have
$$
|L(1/2,\chi)|^2\ll_{\epsilon} \log q\sum_{\substack{N\leq 4q^{1+\epsilon}\\N \text{ dyadic}}}\left|\frac{1}{\sqrt{N}}\sum_{n}\chi(n)V^{\pm}_N(n)\right|^2 +q^{-100},
$$
where $V^{\pm}_N$ is a smooth function depending only on $N$, $q$ and the value of $\chi(-1)=\pm1$, whose support is contained in $[N/4,N]$ and $V_N^{(j)}(x)\ll_j N^{-j}$.
\end{lem}

\section{A short second moment}

Suppose $q=q_1Q_1$ is squarefree, $\chi=\chi_1\psi_1$, where $\chi_1$ and $\psi_1$ are characters of modulus $q_1$ and $Q_1$ respectively. We consider the following second moment

\begin{equation}\label{short2nd}
\mathcal{S}^{\pm}_2(\psi_1):=\sideset{}{{}^{\ast}}\sum_{\substack{\chi_1\rmod{q_1}\\ \chi_1\psi_1(-1)=\pm 1}}|L(1/2,\chi_1\psi_1)|^2,
\end{equation}
and, of course, the full moment 

\begin{equation}\label{combinedmoment}
\mathcal{S}_2(\psi_1):=\mathcal{S}^{+}_2(\psi_1)+\mathcal{S}^{-}_2(\psi_1)
\end{equation}
The reason for which we must consider these two moments separately is because in this way, we will be able to take the weight function $V_N{\pm}$ to be fixed. As it will become clear in our calculation, this will not be of much importance, since we are only looking for upper bounds. By positivity. we can always complete our sums with terms corresponding to the missing characters and find ourselves with a complete sum over characters modulo $q_1$.

The goal of this section is to prove the following result:

\begin{prop}\label{openup+poisson}
Let $q=q_1Q_1$ be squarefree and let $\chi=\chi_1\psi_1$, where $\chi_1$ and $\psi_1$ are characters of modulus $q_1$ and $Q_1$ respectively. Let $K_{\psi_1}$ be as in \eqref{Kchi}. Then for $q_1\ll q^{1/2-\epsilon}$, we have the inequality
$$
\mathcal{S}_2^{\pm}(\psi_1)\ll_{\epsilon} q^{\epsilon}q_1\left\{1+\frac{1}{Q_1^{1/2}}\Bigg|\sum_{0<|m|\ll q^{1+\epsilon}/q_1^2}\alpha^{\pm}_mK_{\psi_1}(mq_1)\Bigg|\right\},
$$
where $\alpha^{\pm}_m$ is a sequence satisfying $\sum_{m}|\alpha^{\pm}_m|^2\ll q^{1+\epsilon}/q_1^2$, 
\end{prop}

We start by noticing that the approximate functional equation (see \eqref{AFE-lem}) and Cauchy-Schwartz show that for every $\epsilon>0$ there exists $N=N_{\epsilon}\leq 4q^{1/2+\epsilon}$ such that

\begin{equation}\label{S<B}
S_2^{\pm}(\psi_1)\ll_{\epsilon}q^{\epsilon}\mathcal{B}^{\pm}(N) +O(q^{-99}),
\end{equation}
where
$$
\mathcal{B}^{\pm}(N)=\frac{1}{N}\sideset{}{{}^{\ast}}\sum_{\substack{\chi_1\rmod{q_1}\\\chi_1\psi_1(-1)=\pm 1}}\left|\sum_{n}\chi_1(n)\psi_1(n)V^{\pm}_N(n)\right|^2.
$$
We now use positivity to extend our sum to all characters modulo $q_1$. Then, by orthogonality of characters, we get
$$
\mathcal{B}^{\pm}(N)\leq \frac{q_1}{N}\sideset{}{{}^{\ast}}\sum_{n,n'}\psi_1(n')\overline{\psi_1(n)}V^{\pm}_N(n')\overline{V^{\pm}_N(n)}.
$$
We now divide the sum on the right according to whether $n=n'$ or not. In the former case the sum is easily seen to be $O(N)$ and for the latter we remark that interchanging the variables $n$ and $n'$ only changes the summand into its complex conjugate. We thus have

\begin{equation}\label{L<<}
\mathcal{B}^{\pm}(N)
\ll q_1+\frac{q_1}{N}\sum_{h\geq 1}\left|\sum_{n\geq 1}\psi_1(n+hq_1)\overline{\psi_1(n)}V^{\pm}_N\left(n+hq_1\right)\overline{V^{\pm}_N\left(n\right)}\right|.
\end{equation}

Let $S^{\pm}_{hq_1}(N;\psi_1)$ denote the inner sum on the right-hand side above. Our next step is to apply Poisson summation to this sum. Before we do so, we observe that the variable $h$ naturally satisfies $h\ll N/q_1$ and that the function $W^{\pm}_{hq_1}$ given by $W^{\pm}_{hq_1}(x):= V^{\pm}_N(x+hq_1)\overline{V^{\pm}_N(x)}$ has support in $[N/4, N]$ and satisfies $W_{hq_1}^{\pm(j)}(x)\ll_j N^{-j}$. Now, by Poisson summation, we have

$$
S_{hq_1}(N;\psi_1)=\frac{N}{Q_1^{1/2}}\sum_{k}K_{\psi_1}(k,hq_1)\widehat{W^{\pm}_{hq_1}}\left(\frac{k}{Q_1}\right),
$$
where $K_{\psi_1}$ is as in \eqref{Kchi}. In view of the properties of $W^{\pm}_{hq_1}$ mentioned in the last paragraph, we see that the function $\widehat{W^{\pm}_{hq_1}}(x)$ is uniformly bounded and decays rapidly for $x\gg N^{-1}q^{\epsilon}$. Hence the contribution from the terms where $|k|>\frac{Q_1q^{\epsilon}}{N}$ to the sum above can certainly be bounded by $q_1$. Thus we see that
\begin{equation}\label{alpha}
\mathcal{B}^{\pm}(N)\ll q_1\left\{1+\frac{1}{Q_1^{1/2}}\left|\sum_{0\leq |k|\ll \frac{Q_1q^{\epsilon}}{N}}\sum_{0<|h|\ll \frac{N}{q_1}}\alpha^{\pm}(h,k)K_{\psi_2}(k;hq_1)\right|\right\},
\end{equation}
where
$$
\alpha^{\pm}(h,k):=N^{-1}\widehat{W^{\pm}_{hq_1}}\left(k/Q_1\right)\ll 1.
$$
We would also like to estimate separately the contribution of the terms where $k=0$. By Lemma \ref{k*l} those are

\begin{equation}\label{k=0}
\ll \frac{q_1}{Q_1}\sum_{0<|h| \ll \frac{N}{q_1}}(h,Q_1)\leq d(Q_1)\frac{N}{Q_1}\ll q_1q^{\epsilon}.
\end{equation}
Again by Lemma \ref{k*l},
$$
K_{\psi_1}(k,hq_1)=\mu(d)R_d(k)K_{\chi_1}(hkq_1),
$$
where $d=(\ell,Q_1)$. Combining this with \eqref{S<B}, \eqref{L<<}, \eqref{alpha} and \eqref{k=0}, we get that

\begin{equation}\label{firstbeta}
\mathcal{S}^{\pm}_2(\psi_1)\ll q^{\epsilon}q_1\left\{1 +\frac{1}{Q_1^{1/2}}\left|\sum_{0<|m|\ll \frac{q^{1+\epsilon}}{q_12}}\alpha^{\pm}_mK_{\psi_1}(mq_1)\right|\right\},
\end{equation}
where 
$$
\alpha^{\pm}_m:=\underset{dh_0k=m}{\sum_{d\mid q}\sum_{k}\sum_{(h_0,q/d)=1}}\mu(d)R_d(k)\alpha^{\pm}(dh_0,k).
$$
In particular, for all $M>0$, we deduce, from Cauchy-Scwarz, the inequality
\begin{equation}\label{beta}
\sum_{0<|m|\ll M}|\alpha^{\pm}_m|^2\ll q^{\epsilon}\sum_{\substack{dh_0k\leq M\\ d\mid q}}(d,k)^2
\leq \sum_{e\mid q}e^2\sum_{m<M/e^2}d_3(m)\ll M^{1+\epsilon}q^{\epsilon}.
\end{equation}
Taking $M=q^{1+\epsilon}/q_1^2$, we conclude the proof of the proposition.

\subsection{The second moment on average}

Let $q=q_1Q_1$ and $\chi=\chi_1\psi_1$ be as in the previous section. Suppose further that $Q_1=q_2q_3$ and $\psi_1=\chi_2\chi_3$. Our next result is an average result for $S_2(\psi_1)$. We will prove the following

\begin{prop}\label{AfterCauchy}
Let $q=q_1q_2q_3$ be squarefree and let $\chi_3$ be a primitive character modulo $q_3$. And let $\mathfrak{X}$ be any set of primitive character modulo $q_2$. Let $S_2(\chi_2\chi_3)$ be as in \eqref{combinedmoment}. Then we have the inequality
$$
\sum_{\chi_2\in \mathfrak{X}}S_2(\chi_2\chi_3)\ll q^{\epsilon}\left\{\left(q_1+\Xi(q,q_1,q_2)\right)X+(q/q_1)^{1/2}X^{1/2}\right\},
$$
with $\Xi(q,q_1,q_2)=q_1^{1/2}q_2^{1/4}$. Moreover, if $q_2$ is $y$-smooth for some $y>0$, and if
\begin{equation}\label{q2>}
q_2>q^{3/2}/q_1^3,
\end{equation}
than we may take
$$
\Xi(q,q_1,q_2)=q^{1/4}q_2^{1/12}y^{1/12}.
$$ 
\end{prop}

It clearly suffices to prove the above bound for $S^{\pm}_2(\chi_2\chi_3)$. We have, from Proposition \ref{openup+poisson}, the inequality
\begin{equation}\label{S<bla+T}
\sum_{\chi_2\in \mathfrak{X}}S_2^{\pm}(\chi_2\chi_3)\ll q^{\epsilon}\left\{q_1X+q_1^{3/2}q^{-1/2}\mathcal{T}^{\pm}(\mathfrak{X};\chi_3)\right\},
\end{equation}
where
$$
\mathcal{T}^{\pm}(\mathfrak{X};\chi_3):=\sum_{\chi_2\in\mathfrak{X}}\left|\sum_{0<|m|\ll q^{1+\epsilon}/q_1^2}\alpha^{\pm}_mK_{\chi_2\chi_3}(mq_1)\right|.
$$
Our next step is to use Cauchy-Schwarz to separate the oscillations of $\alpha^{\pm}_m$ from those of $K_{\chi_2\chi_3}(mq_1)$, but we can still reduce the complexity of the final sum by using the factorization $Q_1=q_2q_3$ and Lemma \ref{lemKchi}. Thus, we obtain
\begin{multline}\label{T<}
\mathcal{T}^{\pm}(\mathfrak{X},\chi_3)\leq
\left(\sum_{0<|m|\ll q^{1+\epsilon}/q_1^2}|\alpha^{\pm}_m|^2\right)^{1/2}\\
\times\left(\sum_{\chi_2,\chi_2'\in\mathfrak{X}}\left|\sum_{0<|m|\ll q^{1+\epsilon}/q_1^2}K_{\chi_2}(mq_1\overline{q_3})\overline{K_{\chi'_2}(mq_1\overline{q_3})}\right|\right)^{1/2}.
\end{multline}
We use \eqref{Lchi} to write
\begin{multline}\label{openwithK0}
\sum_{0<|m|\ll q^{1+\epsilon}/q_1^2}K_{\chi_2}(mq_1\overline{q_3})\overline{K_{\chi'_2}(mq_1\overline{q_3})}\\
=\sum_{\substack{q=r_2s_2\\ \Delta\mid r_2}}\sum_{0<|m|\leq q^{1+\epsilon}/q_1^2}K^{\circ}_{\chi_{2,[r_2]},\chi'_{2,[r_2]}}(mq_1\overline{q_3s_2}),
\end{multline}
where $\Delta=\Delta(\chi_2,\chi_2')$. Now by Lemma \ref{PVKK}, it follows that the inner sum is
$$
\ll q^{\epsilon}\left(qq_1^{-2}r_2^{-1}+r_2^{1/2}\right).
$$
Furthermore, by changing the order of summation, we have
\begin{align*}
\sum_{\chi_2,\chi_2'\in\mathfrak{X}}\sum_{\substack{q=r_2s_2\\ \Delta(\chi_2,\chi'_2)\mid r_2}}\frac 1{r_2}&\leq \sum_{\chi_2\in\mathfrak{X}}\sum_{r_2\mid q_2}\frac{1}{r_2}\#\left\{\chi_2'\rmod{q_2};\Delta(\chi_2,\chi_2')\mid r_2\right\}\\
&\leq \sum_{\chi_2\in\mathfrak{X}}\sum_{r_2\mid q_2}\frac{1}{r_2}\#\left\{\chi'\rmod{r_2}\right\}\notag\\
&\leq d(q)X\notag.
\end{align*}
Inserting the last estimates in \eqref{T<}, we see that
$$
\mathcal{T}^{\pm}(\mathfrak{X};\chi_3)\ll q^{\epsilon}\left(qq_1^{-2}X^{1/2}+q^{1/2}q_1^{-1}q_2^{1/4}X\right).
$$
Replacing this in \eqref{S<bla+T} gives the proposition with $\Xi(q,q_1,q_2)=q_1^{1/2}q_2^{1/4}$.\\

For the other value of $\Xi(q,q_1,q_2)$, we now suppose that $q_2$ is $y$-smooth and that $q_2> q^{3/2}/q_1^3$. The same argument as before, shows that the contribution of the terms where $r_2\leq q^{3/2}/q_1^3$ to \eqref{T<} is bounded by
$$
q^{\epsilon}\left(qq_1^{-2}X^{1/2}+q^{7/8}q_1^{-7/4}X\right).
$$
And by Lemma \ref{vdCKK}, the remaining terms contribute
$$
\ll q^{\epsilon}Xq^{3/4}q_1^{-3/2}q_2^{1/12}y^{1/12}.
$$
Putting these bounds together, and using again that $q_2>q^{3/2}q_1^{-3}$, we obtain
$$
\mathcal{T}^{\pm}(\mathfrak{X};\chi_3)\ll q^{\epsilon}\left(qq_1^{-2}X^{1/2}+q^{3/4}q_1^{-3/2}q_2^{1/12}y^{1/12}X\right).
$$
Replacing this in \eqref{S<bla+T} gives the proposition with the second choice for $\Xi(q,q_1,q_2)$.
\section{proof of Theorem \ref{large}}\label{proofofthm}
We recall that we shall prove four different bounds for $\mathcal{R}(V;q)$. The bound \eqref{A1} follows from the classical result
$$
\sideset{}{{}^{\ast}}\sum_{\chi\rmod q}|L(1/2,\chi)|^4\ll_{\epsilon} q^{1+\epsilon}.
$$
We now prooceed to prove \eqref{A2}. Notice that we can suppose that $V\geq q^{1/8}$, since otherwise, the bound \eqref{A1} is stronger. We can always suppose that $\delta$ is as smal as we want by simply taking the implied constant large enough so that the results become trivial for $\delta$ after some point. We start by supposing we can factorise $q$ as $q=q_1q_2q_3$, where the $q_i$ are squarefree and relatively prime.

We write $\mathcal{R}:=\mathcal{R}(V;q)$ and for a character $\chi_3$ of modulus $q_3$, we let $\mathcal{R}_3(\chi_3)=\{\chi\pmod{q/q_3};\chi\chi_3\in\mathcal{R}\}$. Finally, we let
$$
\mathfrak{X}(\chi_3):=\{\chi_2\rmod{q_2};\chi_1\chi_2\in \mathcal{R}(\chi_3),\text{ for some }\chi_1\rmod{q_1}\}.
$$
We have the trivial inequality
$$
R(\chi_3):=\#\mathcal{R}(\chi_3)\leq V^{-2}\sum_{\chi_2\in\mathfrak{X}(\chi_3)}S_2(\chi_2\chi_3).
$$
Thus, by Proposition \ref{AfterCauchy} with the first choice for $\Xi(q,q_1,q_2)$, we see that

$$
R(\chi_3)\ll_{\epsilon} q^{\epsilon}V^{-2}\left\{\left(q_1+q_1^{1/2}q_2^{1/4}\right)R(\chi_3)+(q/q_1)^{1/2}R(\chi_3)^{1/2}\right\}.
$$
Since $q$ is $q^{\delta}$-smooth and we are in the case where $V>q^{\delta}$, we can always pick $q_1$ and $q_2$ satsfiyng the inequalities
\begin{equation}\label{q1q2}
\begin{cases}
V^2q^{-2\delta}\leq q_1\leq V^2q^{-\delta},\\
V^4q^{-2\delta}\leq q_2\leq V^4q^{-\delta}.
\end{cases}
\end{equation}
In this case we can ignore the first two terms by taking any $\epsilon<\delta$. It then follows that
$$
R(\chi_3)\ll_{\delta} q^{\delta}V^{-4}qq_1^{-1}.
$$
Summing over $\chi_3$ one gets
$$
\#\mathcal{R}(V,q)\ll_{\delta} q^{\delta}V^{-4}q^2q_1^{-2}q_2^{-1}\leq q^{2+7\delta}V^{-12}.
$$

Now, finally, we prove the bounds \eqref{A3} and \eqref{A4}. We suppose $V>q^{3/20+\delta}$ and suppose again that we have $q=q_1q_2q_3$ with the $q_i$ squarefree and relatively prime. Suppose further that $q_2$ is $q^{\delta}$-smooth and $q_2>q^{3/2}q_1^{-3}$.

As in the previous case, we use Proposition \ref{AfterCauchy} but this time with the second choice of $\Xi(q,q_1,q_2)$. Thus, with the notation as above, we obtain that

\begin{equation}\label{R<}
R(\chi_3)\ll_{\epsilon} q^{\epsilon}V^{-2}\left\{\left(q_1+q^{1/4+\delta/12}q_2^{1/12}\right)R(\chi_3)+(q/q_1)^{1/2}R(\chi_3)^{1/2}\right\}.
\end{equation}
As before, we want to get rid of the first two terms. We use $q^{\delta}$-smoothness to take $q_1$ and $q_2$ satisfying
$$
\begin{cases}
V^2q^{-2\delta}\leq q_1\leq V^2q^{-\delta},\\
Wq^{-\delta}\leq q_2\leq W,
\end{cases}
$$
where $W:=\min(V^{24}q^{-3}q^{-12\delta},q/q_1)$. Notice that the condition $V>q^{3/20+\delta}$ and $\delta$ sufficientely small ensures that
$$
Wq^{-\delta}>q^{3/2}q_1^{-3}.
$$
Hence \eqref{R<} holds for this choice of $q_1$ and $q_2$ and moreover, we can ignore the first two terms. We deduce by taking $\epsilon <\delta$
$$
R(\chi_3)\ll q^{\delta}V^{-4}qq_1^{-1}.
$$
Summing over $\chi_3$ one gets
$$
\#\mathcal{R}(V,q)\ll_{\delta} q^{\delta}V^{-4}q^2q_1^{-2}q_2^{-1}\ll_{\delta} q^{O(\delta)}\max(qV^{-6},q^{5}V^{-32}),
$$
which concludes the proof of \eqref{A3} and \eqref{A4}, when $V>q^{3/20+\delta}$.\\

Finally, if $q^{3/20}< V\leq q^{3/20+\delta}$, we can use \eqref{A2} and deduce that for some $C>0$,
$$
\#\mathcal{R}(V,q)\ll_{\delta} q^{2+C\delta}V^{-12}\leq q^{5+(C+20)\delta}V^{-32}.
$$
The proof of Theorem \ref{large} is now finished.

\section{Bounds on complete exponential sums}\label{tracefunctions}

In this section we prove Proposition \ref{completesums} on estimates for complete exponential sums related to $K_{\chi}$ (recall definition \eqref{Kchi}).

The upper bound \eqref{PV} follows by developing the left-hand side according to definitions \eqref{Kchi} and \eqref{Kcirc} and using the Weil bound to the sum

$$
\sum_{\substack{x\rmod p\\ x\neq 0,t}}\chi(1+\overline{x})\overline{\chi'}(1+\overline{x-t}),
$$

when $t\neq0$ and an explicit calculation if $t=0$.

As for \eqref{vdC}, this is much more involved and to prove it we need to invoke the formalism of $\ell$-adic sheaves and its trace functions of Deligne. The proof is inspired by the work of Fouvry-Kowalski-Michel and draws upon material developed by Katz in his books \citep{Katz1988gauss} and \citep{Katz1990exponential}. We refer the reader to \citep[section 6]{Polymath} for a nice introduction to the concepts used in this section.

We start with the case $\chi\neq\chi'$. We want to estimate

\begin{equation}\label{correlation}
\sum_{u\rmod p}K_{\chi}(s+u)\,\overline{K_{\chi'}(s+u)}\,\overline{K_{\chi}(u)}\,K_{\chi'}(u)e\left(-\frac{tu}{p}\right),
\end{equation}
where $K_{\chi}$ is given by \eqref{Kchi}. The case $t=s=0$ follows from the Weil bound (Lemma \ref{Kchi}). For the remaining cases, we shall use the techniques from \citep{FKM2015study} in the same spirit of \citep[Appendix]{Irving2016estimates}.

We want to apply \citep[Theorem 2.7]{FKM2015study} and \citep[Proposition 1.1]{FKM2015study} to the following data:
\begin{itemize}
\item The sheaf $\mathcal{G}=\mathcal{L}_{\psi(tX)}$, with $\psi$ the additive character corresponding to $x\mapsto e(x/p)$,

\item The sheaves $\mathcal{F}_1=\mathcal{F}_{\chi},\;\mathcal{F}_2=\mathcal{F}_{\chi'},\,\mathcal{F}_3=[+s]^{\ast}\mathcal{F}_{\chi},\,\mathcal{F}_4[+s]^{\ast}\mathcal{F}_{\chi'},$ where for a primitive dirichlet character $\chi$, $\mathcal{F}_{\chi}=FT_{\psi}(\mathcal{H}_{\chi})\otimes\mathcal{L}_{\psi}(X/2)$ with $\mathcal{H}_{\chi}$ the Kummer sheaf $\mathcal{L}_{\chi(1+X^{-1})}$

\item The open set $U=\Aa^1-\{0,-1,-s,-s-1\}$.
\end{itemize}

We first remark that the trace function associated to $\mathcal{F}_{\chi}$ is not exactly $K_{\chi}(u)$ but $K_{\chi}(u)e(\overline{2}u/p)$ instead. However this has no influence in the study of the correlation sum \eqref{correlation}. The reason why we made the twist by an Artin-Schreier sheaf is that in this way we obtain a self-dual sheaf (see \citep[Appendix]{Irving2016estimates}).

By the same argument as in \citep[Appendix, Proposition 3]{Irving2016estimates}, if the hypothesis of \citep[Theorem 2.7]{FKM2015study} are satisfied, then we have that the sum in \eqref{correlation} is $\ll p^{1/2}$ unless $s=t=0$. We shall now verify that that the above data satisfies these conditions. This is almost entirely done in the proof of \citep[Appendix, Proposition 3]{Irving2016estimates}. For instance, in order to prove that the $\mathcal{F}_i$ form an $U$-generous tuple, we are only left with proving that for $i\neq j$, one cannot have
$$
\mathcal{F}_i\simeq \mathcal{F}_j\otimes \mathcal{L},
$$
for some lisse sheaf of rank one in $U$. Moreover, by looking at the tame ramification of $\mathcal{F}_{\chi}$ at $0$, as in \citep[Appendix]{Irving2016estimates} we can reduce it to considering the cases $\{i,j\} =\{1,2\}$ or $\{3,4\}$. Finally, by making an additive shift, we can focus only on the first of these two. In other word we have to rule out the possibility of a geometric isomorphism
\begin{equation}\label{1dimtwist}
\mathcal{F}_{\chi}\simeq \mathcal{F}_{\chi'}\otimes\mathcal{L}.
\end{equation}
Such an $\mathcal{L}$, if it exists, must be unramified at $\Gg_m$ and, since $\mathcal{F}_{\chi}$ and $\mathcal{F}_{\chi'}$ at $0$ are formed by a single unipotent block (this comes from \citep[Theorem 7.5.4 (7)]{Katz1990exponential} and the fact that the Kummer sheaf $\mathcal{H}_{\chi}$ is unramified at $+\infty$), we see that $\mathcal{L}$ must be geometrically trivial at $0$. The situation at $\infty$ is a bit more delicate. By using Laumon's local Fourier Transform (see \citep[Theorem 7.4.4 (2)]{Katz1990exponential} and \citep[Corollary 7.4.1 (3)]{Katz1990exponential}), we know that as an $I(\infty)$ representation, we have
\begin{equation}\label{F(oo)}
\mathcal{F}_{\chi}(\infty)\approx \left(\mathcal{L}_{\overline{\chi}}\otimes\mathcal{L}_{\psi(X/2)}\right)\oplus\left(\mathcal{L}_{\chi}\otimes\mathcal{L}_{\psi(-X/2)}\right)
\end{equation} 
and a similar decomposition for $\chi'$. Since both parts are $1$-dimensional, they must be irreducible. Thus, taking the tensor product of $\mathcal{F}_{\chi'}(\infty)$ with $\mathcal{L}(\infty)$, it must send each of the two pieces of $\mathcal{F}_{\chi'}(\infty)$ in one of the pieces in \eqref{F(oo)}. By analysing the two possibilities, we find out that $\chi\overline{\chi'}$ must be quadratic and $\mathcal{L}(\infty)=\mathcal{L}_{\chi'\overline{\chi}}$. In particular, $\mathcal{L}$ is tame at $\infty$.

Summarizing, we have seen that such $\mathcal{L}$, if it exists, must be everywhere unramified, except maybe at $\infty$ where it is at most tame. We claim that this forces $\mathcal{L}$ to be geometrically trivial. Indeed, suppose that $\mathcal{L}$ is has a tame ramification at $\infty$. Then, by an application of the Grothendieck-Ogg-Shafarevich formula (see \citep[2.3.1]{Katz1988gauss}),
% since $\mathcal{L}$ is 1-dimensional and not geometrically trivial,
one obtains that
$$
\operatorname{dim}H^1_c(\Aa\otimes\overline{\Ff_p},\mathcal{L})=-\chi_c(\Aa^1\otimes\overline{\Ff_p})=-1,
$$
which is impossible. Finally, since $\mathcal{F}_{\chi}$ and $\mathcal{F}_{\chi'}$ are not geometrically isomorphic, the geometric isomorphism \eqref{1dimtwist} is impossible. 

We shall now consider the diagonal case (\textit{i.e.} $\chi=\chi'$). The case $t\neq 0$ also follows from the previous argument, so we can focus on the case $t=0$. We must then estimate the sum
\begin{multline}\label{openKo}
\sum_{u\rmod p}(|K_{\chi}(s+u)|^2-1)\,(|K_{\chi}(u)|^2-1)=\\
\sum_{u\rmod p}|K_{\chi}(s+u)|^2\,|K_{\chi}(u)|^2-p+4,
\end{multline}
where, here, we used the elementary calculation $\sum_{u(p)}|K_{\chi}(u)|^2=p-2$.

We apply \citep[Corollary 1.7]{FKM2015study} to the sheaf $\mathcal{F}_{\chi}$. As a part of the proof of \citep[Proposition A.3]{Irving2016estimates}, it was shown that the sheaf is bountiful of $\mathrm{SL}_2$-type (see \citep[Definition 1.2]{FKM2015study}). So we only need to ensure that the arithmetic monodromy group is also $\mathrm{SL}_2$. But this follows because $\mathcal{F}_{\chi}$ is (aritmetically) self-dual. Hence, the result applies and gives

$$
\sum_{u\rmod p}|K_{\chi}(s+u)|^2\,|K_{\chi}(u)|^2=Ap+O(p^{1/2}),
$$
where
$$
A=
\begin{cases}
A(2)^2,\text{ if }s\neq 0\\
A(4),\text{ if }s=0,
\end{cases}
$$
and $A(i)$ denotes the multiplicity of the trivial representation of $\mathrm{SL}_2$ in the $i$-th tensor product of the standard representation of $\mathrm{SL}_2$. By character theory of its compact real form $SU(2)$,
$$
A(i)=\frac{2}{\pi}\int_{0}^{\pi}(2\cos(\theta))^i\sin(\theta)d\theta.
$$
In particular, we have $A(2)=1$ and $A(4)=2$. Replacing this result in \eqref{openKo} concludes the proof of Proposition \ref{completesums}.

\bibliographystyle{plain}
%\bibliography{references}
%\input{12thmoment.bbl}

\end{document}